\documentclass[11pt]{amsart}%
\usepackage[a4paper]{geometry}
\usepackage{enumerate,color}
\usepackage{hyperref}
\usepackage{amsthm,amsmath,amsfonts}
\usepackage{amsfonts,amsthm,amssymb,amsmath}
\usepackage{graphicx,color,mathrsfs}
\usepackage{hyperref}
\usepackage{hypernat}
\usepackage[]{mcode}

\usepackage{amsmath,epsfig,graphicx,color}
\definecolor{darkblue}{rgb}{.2, 0.2,.8}
\definecolor{darkgreen}{rgb}{0,0.5,0.3}
\definecolor{darkred}{rgb}{.8, .1,.1}
\newcommand{\Pro}{\mathbb{P}}

\newcommand{\blue}{}

\newcommand{\ex}{{\rm e}\,}

\newtheorem{lemma}{Lemma}[section]

\newtheorem{theorem}[lemma]{Theorem}

\newtheorem{proposition}[lemma]{Proposition}
\newtheorem{definition}[lemma]{Definition}
\newtheorem{corollary}[lemma]{Corollary}
\newtheorem{example}[lemma]{Example}
\newtheorem{exercise}[lemma]{Exercise}
\newtheorem{remark}[lemma]{Remark}
\newtheorem{fig}[lemma]{Figure}
\newtheorem{tab}[lemma]{Table}

\newcommand{\bth}{\begin{theorem}}
\newcommand{\ethe}{\end{theorem}}

\newcommand{\bre}{\begin{remark}\em }
\newcommand{\ere}{\end{remark}}

\newcommand{\ble}{\begin{lemma}}
\newcommand{\ele}{\end{lemma}}

\newcommand{\pp}{point process}
\newcommand{\bde}{\begin{definition}}
\newcommand{\ede}{\end{definition}}
\newcommand{\bco}{\begin{corollary}}
\newcommand{\eco}{\end{corollary}}

\newcommand{\bpr}{\begin{proposition}}
\newcommand{\epr}{\end{proposition}}

\newcommand{\bexer}{\begin{exercise}}
\newcommand{\eexer}{\end{exercise}}

\newcommand{\bexam}{\begin{example}}
\newcommand{\eexam}{\end{example}}

\newcommand{\bfi}{\begin{fig}}
\newcommand{\efi}{\end{fig}}

\newcommand{\btab}{\begin{tab}}
\newcommand{\etab}{\end{tab}}

\newcommand{\fidi}{finite-dimensional distribution}
\newcommand{\rv}{random variable}

\newcommand{\rhs}{right-hand side}
\newcommand{\df}{distribution function}

\newcommand{\beao}{\begin{eqnarray*}}
\newcommand{\eeao}{\end{eqnarray*}\noindent}

\newcommand{\beam}{\begin{eqnarray}}
\newcommand{\eeam}{\end{eqnarray}\noindent}

\newcommand{\beqq}{\begin{equation}}
\newcommand{\eeqq}{\end{equation}\noindent}

\newcommand{\bce}{\begin{center}}
\newcommand{\ece}{\end{center}}

\newcommand{\barr}{\begin{array}}
\newcommand{\earr}{\end{array}}

\newcommand{\eqd}{\stackrel{d}{=}}

\newcommand{\vague}{\stackrel{\lower0.2ex\hbox{$\scriptscriptstyle
                    \it{v} $}}{\rightarrow}}
\newcommand{\weak}{\stackrel{\lower0.2ex\hbox{$\scriptscriptstyle
                    \it{w} $}}{\rightarrow}}
\newcommand{\what}{\stackrel{\lower0.2ex\hbox{$\scriptscriptstyle
                    \it{\hat{w}} $}}{\rightarrow}}

\newcommand{\bdis}{\begin{displaymath}}
\newcommand{\edis}{\end{displaymath}\noindent}

\newcommand{\nto}{n\to\infty}

\newcommand{\bbr}{{\mathbb R}}

\newcommand{\bbz}{{\mathbb Z}}

\newcommand{\BM}{Brownian motion}

\newcommand{\evd}{extreme value distribution}

\newcommand{\wrt}{with respect to}

\newcommand{\fct}{function}

\newcommand{\ds}{distribution}

\newcommand{\rep}{representation}

\newcommand{\seq}{sequence}

\newcommand{\pro}{probabilit}

\newcommand{\ms}{measure}

\newcommand{\R}{\mathbb{R}}

\newcommand{\E}{\mathbb{E}}

\newcommand{\Q}{\mathbb{Q}}
\newcommand{\Expec}{\mathbb{E}}
\newcommand{\Var}{\mathbb{V}\mathrm{ar}}
\newcommand{\Cov}{\mathbb{C}\mathrm{ov}}

\begin{document}

\title{Exact simulation of Brown-Resnick random fields {\blue at a finite number of locations}}
\author{A.B.~Dieker}
\address{H. Milton Stewart School of Industrial and Systems Engineering, 
Georgia Institute of Technology, Atlanta, GA 30332, U.S.A.}
\email{ton.dieker@isye.gatech.edu}
\address{Industrial Engineering and Operations Research, Columbia University, New York, NY 10027, U.S.A.}
\email{dieker@columbia.edu}
\author{T.~Mikosch}
\address{University of Copenhagen, Department of Mathematics, 2100
  Copenhagen, Denmark}
\email{mikosch@math.ku.dk}

\date{\today} 

\begin{abstract} We propose an exact simulation method for Brown-Resnick random fields,
building on new representations for these stationary 
max-stable fields.  The main idea is to apply suitable 
changes of measure.
%These are stationary max-stable random fields which have found various applications in the recent past.
  
\end{abstract}


\maketitle

\section{Introduction}
Max-stable random fields are fundamental {\blue models} for spatial extremes.
These models have been coined by de Haan \cite{dehaan:1984}, and have recently found applications to 
extreme meteorological events such as rainfall modeling and extreme temperatures 
(Buishand et al.~\cite{buishand:dehaan:zhou:2008}, de Haan and Zhou \cite{dehaan:zhou:2008},
Dombry et al.~\cite{dombry:minko:ribatet:2013}, Davis et al.~\cite{davis:kluppelberg:steinkohl:2013}, Huser and Davison \cite{huser:davison:2014}). 
%These processes are considered the analogs of Gaussian processes in the world of extremes. 
There are three different kinds of {\blue normalized} 
max-stable processes, with Gumbel, Fr\'echet, and Weibull marginals, respectively.
In what follows, we restrict ourselves to max-stable processes with Gumbel marginals; 
corresponding results for Fr\'echet and Weibull marginals can be
obtained by a monotone transformation of the Gumbel case. 
\par
This paper studies a particular class of max-stable random fields known as {\em Brown-Resnick random fields}.
Simulation of these and related processes is complicated, and 
%an extensive  literature has been devoted to approximate simulation;
{\blue the literature exclusively focuses on approximate simulation techniques;}
see for example Schlather \cite{schlather:2002}, 
Oesting et al.~\cite{oesting:kabluchko:schlather:2012}, Engelke et al.~\cite{engelke:kabluchko:schlather:2011},
Oesting and Schlather~\cite{oesting:schlather:2014}, Dombry et al.~\cite{dombry:minko:ribatet:2013}.
\par
This paper is the first to devise an {\em exact} simulation method for Brown-Resnick random fields.
The key ingredient is a new representation for Brown-Resnick random fields, which is of independent interest.
In fact, we show that there is an uncountable {\em family} of representations.
At the heart of our derivation of these representations lies a change of measure argument.
\par
We now describe the results in this paper in more detail.
For some index set $T\subset \bbr^d$, the  process
$(Y(t))_{t\in T}$ of {\blue real-valued} \rv s is {\em max-stable} (with
Gumbel marginals) if 
for a \seq\ of iid copies $(Y^{(i)}(t))_{t\in T}$, $i=1,2,\ldots$, 
of $(Y(t))_{t\in T}$ the following relation holds
\beao
\big(\max_{i=1,\ldots,n} Y^{(i)}(t)-\log n \big)_{t\in T} \eqd (Y(t))_{t\in
  T}\,,\quad n  \ge 1  \,,
\eeao
where this relation is interpreted in the sense of equality of the \fidi s. 
Then, in particular, all one-dimensional marginals of the process 
 $(Y(t))_{t\in T}$ are Gumbel distributed, i.e., $Y(t)$ has \df\ 
$\Lambda(x -c(t))= \exp(-\ex^{-(x-c(t))})$, $x\in\bbr$, for some \fct\ $c(t)\in \bbr$, $t\in T$. 
Throughout this paper, we work with $T=\R^d$.
\par 
In this paper, we consider a class of max-stable processes with representation
\beam\label{eq:BrownResnick}
\eta(t) = \sup_{i\ge 1} \big(V_i +W_i(t)-\sigma^2(t)/2\big)\,,\qquad t\in \bbr^d\,,
\eeam
where $\sigma^2(t)=\Var(W_1(t))$,
$t\in \bbr^d$, $(W_i)$ is a \seq\ of iid centered Gaussian processes 
with stationary increments on $\bbr^d$, 
and $(V_i)$ are the points of a Poisson process on $\bbr$
with intensity \ms\ $\ex^{-x}\,dx$. In the case of \BM s $\blue (W_i)$, 
the process \eqref{eq:BrownResnick} was considered by Brown and Resnick
\cite{brown:resnick:1977} and shown to be stationary. It is common to
refer to the more general model  \eqref{eq:BrownResnick} as {\em Brown-Resnick
random field} as well. 
\par
The representation (\ref{eq:BrownResnick}) is not particularly suitable for exact sampling.
Although $(V_i)$ and $(W_i)$ are
easily simulated, it turns out that the naive 
simulation approach of replacing $\sup_{i\ge 1}$ by $\sup_{i\le N}$ for some large $N$, may fail.
For example, assume  that $W_i$ is standard \BM\ on $\bbr$. Then, in
view of the law of the iterated logarithm, each of the
processes $\blue W_i(t)-\sigma^2(t)/2$ drifts to $-\infty$ a.s. as $t\to\infty$. In turn,
the process $\sup_{i\le N} \big(V_i +W_i(t)-\sigma^2(t)/2\big)$ drifts to $-\infty$
as $t\to\infty$ as well. In particular, the simulation of $\eta$ requires an increasing number $N$ 
if one aims at a sample path of the process on a larger interval. 
More importantly, it is unclear how $N$ should be chosen.
\par
Using our new representations, we obtain an exact sampling method for 
$\eta$ at the points $t_1,\ldots,t_n\in \bbr^d$, meaning that 
the output of the method has the same distribution as
$(\eta(t_1),\ldots,\eta(t_n))$. 
In our method, it is no longer problematic that the 
processes $W_i(t)-\sigma^2(t)/2$ drift away to $-\infty$
{\blue and a truncation point is automatically identified by our algorithm}.
\par
Several properties of Brown-Resnick processes readily follow from our representations, 
although they are not straightforward to see from (\ref{eq:BrownResnick}).
For instance, the process $\eta$ is stationary in the sense that $\eta$ has the
same \ds\ as $\eta(\cdot +c)$ for any choice of $c\in \bbr^d$. 
{\blue The process $\eta$ also has standard Gumbel marginals.} 
In deriving our representations from (\ref{eq:BrownResnick}),
$\sigma^2$ drops out and we recover the known fact that the law of $\eta$ only depends on the {\em variogram}
\[
\gamma(t) = \frac 12 \Expec(W(t)-W(0))^2\,,\qquad t\in \R^d.
\]
These properties were proved in Kabluchko et al.~\cite{kabluchko:schlather:dehaan:2009} with arguably
more elaborate techniques. 
\par
Given an exact simulation method for the Brown-Resnick process 
with Gumbel marginals, we also have an exact simulation method for 
this process with Fr\'echet or Weibull marginals. For
example, the processes $\ex^{\eta}$ and $-\ex^{-\eta}$
have  Fr\'echet $\Phi_1(x)=\ex^{-x^{-1}}$, $x>0$, and Weibull
$\Psi_1(x)=\ex^{-|x|}$, $x<0$, marginals, respectively.

\subsection*{Notation}
We use the symbols $W_1,W_2,\ldots$ for iid centered Gaussian random fields with 
stationary increments, variance function $\sigma^2$, and variogram $\gamma$.
We use the symbols $Z_1,Z_2,\ldots$ for iid Gaussian random fields with stationary increments,
mean function $-\gamma$, variance function $2\gamma$, variogram $\gamma$, 
{\blue and vanishing at the origin}.
A generic copy of these fields is denoted by $W$ and $Z$, respectively.
%For such processes, we have
%\[
%\Expec(\exp(Z(t))) = 1\,, \qquad t\in \R^d.
%\]

\section{Representations}
In this section we provide new representations for the 
Brown-Resnick random field $\eta$ given in \eqref{eq:BrownResnick}.
These representations arise from a change of measure.
We make the same assumptions on the stationary Brown-Resnick process 
as in the previous section. 
All proofs for this section are in Section~\ref{sec:proofs}.
We fix the functions $\sigma^2$ and $\gamma$ throughout this section.

The following theorem is the main result of this section.

\bth
\label{thm:main}
Suppose we are given an arbitrary \pro y \ms\ $\mu$ on $\R^d$.
Consider 
\[
\zeta(t)= \sup_{i\ge 1} 
\Big(V_i + Z_i(t-T_i) - \log\big(\int_{\R^d} \exp\big(Z_i(s-T_i) \big)\,\mu(ds)\big)\Big)\,,\qquad t\in \R^d\,,
\]
where $\blue \big((T_i,V_i)\big)_{i\ge 1}$ are the points of a Poisson
process on $\blue \R^d\times \R$
with intensity measure 
$\blue \mu(dt)\times \ex^{-v}dv$.
Then the random fields $(\eta(t))_{t\in \R^d}$ and $(\zeta(t))_{t\in \R^d}$ have the same distribution.
\ethe

\bre
There is a continuum of random fields with the same distribution as $\eta$, one for each measure $\mu$.
\ere
\bre
{\blue 
Under the assumptions of Theorem~\ref{thm:main}, it is known that
the field $(\zeta'(t))_{t\in\R^d}$ with
\[
\zeta'(t)= \sup_{i\ge 1}  \Big(V_i + Z_i(t-T_i) \Big) \,,\qquad t\in \R^d
\]
also has the same distribution as $(\eta(t))_{t\in\R}$.
Although the fields $\zeta$ and $\zeta'$ differ due to the additional 
log-term, the theorem states they have the same distribution.
This surprising fact becomes perhaps more plausible after noting that, for every $i$,
\[
\log \Expec \big(\int_{\R^d} \exp\big(Z_i(s-T_i) \big)\,\mu(ds)\big)=0.
\]
}
\ere
\bre
\label{rem:simrep}
If $\sigma^2/2=\gamma$, then $(W_i(t-T_i))$ has the same 
distribution as $(W_i(t)-W_i(T_i))$.
The term $W_i(T_i)$ drops out of the expression for $\zeta$, so in that case the random field
\[
\sup_{i\ge 1}  \Big(V_i + W_i(t) - \gamma(t-T_i) - \log\big(\int_{\R^d} \exp\big(
W_i(s) -\gamma(s-T_i)\big)\mu(ds)\big)\Big)\,,\qquad t\in \R^d\,,
\]
also has the same distribution as $(\eta(t))_{t\in \R^d}$.
\ere
\bre 
{\blue Oesting et al.~\cite{oesting:kabluchko:schlather:2012,oesting:schlather:zhou:2013}} provided various alternative \pp\ \rep s of Brown-Resnick random fields.
These \rep s are different from ours, {\blue although they appear similar in spirit.}
{\blue The paper \cite{oesting:kabluchko:schlather:2012}}
proposed to introduce random time shifts of the processes $W_i$
and used this idea to derive approximate sampling methods for $\eta$. 
{\blue The paper \cite{oesting:schlather:zhou:2013} focused on a
much wider class of max-stable processes than this paper.}
\ere

Theorem~\ref{thm:main} leads to the following three well known facts
proved in Kabluchko et al.~\cite{kabluchko:schlather:dehaan:2009}. 
\begin{corollary}
\label{cor:stationary}
The field $\eta$ is stationary.
\end{corollary}
\begin{proof}
Let $\mu$ be a Dirac point mass at some arbitrary $t^*\in\R^d$. Theorem~\ref{thm:main} implies that the random field
$(\sup_{i\ge 1} (V_i+Z_i(t-t^*)))_{t\in \R^d}$ has the same distribution as $(\eta(t))_{t\in \R^d}$.
In particular, the distribution does not depend on $t^*$.
\iffalse
By the change of variable formula, we have for any $i\ge 1$,
\begin{eqnarray*}
\lefteqn{Z_i(t-T_i) -\log\big(\int_S \exp\big(Z_i(s-T_i)\big)\,\mu(ds)\big)}\\ &=& 
Z_i(t^*+t-(t^*+T_i)) -\log\big(\int_S \exp\big(Z_i(t^*+s-(t^*+T_i))\big)\,\mu(ds)\big)\\
&=& 
Z_i(t^*+t-T_i^*) -\log\big(\int_{S^*} \exp\big(Z_i(s-T_i^*)\big)\,\mu^*(ds)\big),
\end{eqnarray*}
where $T^*_i=t^*+T_i$, $S^*=t^*+S$, and $\mu^*(ds)=\mu(ds-t^*)$ is the distribution of $T^*_i$.
Applying Theorem~\ref{thm:main} to the left-hand side (after adding $V_i$ and taking the supremum over $i$)
shows that the left-hand side has the same distribution as $(\eta(t))_{t\in S}$ $\mu$-almost everywhere.
Applying Theorem~\ref{thm:main} to the right-hand side (again after adding $V_i$ and taking the supremum over $i$)
yields that the right-hand side has the same distribution as $(\eta(t))_{t\in S^*}$ $\mu$-almost everywhere.
Thus the distributions of $(\eta(t))_{t\in S}$ and
$(\eta(t+t^*))_{t\in S}$ are equal $\mu$-almost everywhere.
In particular, this corollary yields the equality of all
finite-dimensional distributions upon choosing discrete measures $\mu$.\fi
\end{proof}

\begin{corollary}
The one-dimensional marginals of $(\eta(t))_{t\in\R^d}$ have the Gumbel distribution.
\end{corollary}
\begin{proof}
If we let $\mu$ be a point mass as in the proof of the preceding corollary,
then we find that $\eta(t)$ has the same distribution as $\sup_{i\ge 1} V_i$ 
for every $t\in \R^d$.
\end{proof}

\begin{corollary}
The distribution of $(\eta(t))_{t\in\R^d}$ only depends on the variogram $\gamma$.
\end{corollary}
\begin{proof}
Since the processes $Z_i$ are completely determined by $\gamma$, the law of $(\zeta(t))_{t\in \R^d}$
depends only on $\gamma$. Theorem~\ref{thm:main} therefore immediately yields the claim.
\end{proof}

There are some interesting connections between Brown-Resnick random fields and familiar quantities in 
extreme value theory, which simply follow from the known finite-dimensional distribution functions of such fields.
Details on this distribution function can be found in Section~\ref{sec:proofs} (specifically Lemma~\ref{lem:suplemma});
for now, we note that
%If we make the set of $t_i$'s dense in any cube $[0,N]^d$, $N>0$, and 
if $\eta$ is stochastically continuous we have, for any $N>0$,
\beao
\Pro \Big(\sup_{t\in [0,N]^d} \eta(t)\le x\Big)
&=& \exp\Big(-\ex^{-x}\,\Expec \exp\big(\sup_{t\in [0,N]^d}
Z(t)\big)\Big),\qquad x\in\bbr\,,
\eeao
and therefore
\beao
\Pro \Big(\sup_{t\in [0,N]^d} \eta(t)-d\log N \le x\Big)
&=& \exp\Big(-\ex^{-x}\,N^{-d}\,\Expec \exp \big(\sup_{t\in [0,N]^d}
Z(t)\big)\Big) \,,\qquad x\in\bbr\,.
 \eeao
Dieker and Yakir~\cite[Cor.~1]{dieker:yakir:2013}, show
that the set function
\[
f(A) = \Expec \exp \big(\sup_{t\in A} Z(t)\big)\,, \qquad A\subset \R^d
\]
is translation invariant: $f(A)= f(t+A)$ for $t\in\R^d$; {\blue
this also follows from Corollary~\ref{cor:stationary}}. 
(They only write out the one-dimensional {\blue fractional
Brownian motion} case, but the {\blue more general}
case follows from exactly the same arguments; it is based on Lemma~\ref{lem:changeofmeasure} below.)
Moreover, $f$ is subadditive in the sense that $f(A_1\cup A_2)\le f(A_1)+f(A_2)$ for disjoint subsets $A_1,A_2\subset\R^d$.
A basic fact about such functions (e.g., Xanh \cite{nguyen:1979}) is that $f(A)$ grows like the Lebesgue measure of $A$ for large sets $A$.
In particular, this result implies that the limit
\beao
\lim_{N\to\infty}N^{-d}\,\Expec \exp \Big(\sup_{t\in [0,N]^d} Z(t)\Big)
\eeao
exists. This quantity is known as {\em Pickands's constant}; {\blue we
  refer to the monograph by Piterbarg \cite{piterbarg:1996} for an
  extensive discussion of these quantities.}
The numerical determination of this constant and
the simulation of the Brown-Resnick process $\eta$ suffer from the
same problems mentioned in the Introduction.  Dieker and Yakir~\cite{dieker:yakir:2013} proposed a Monte Carlo method for
determining the Pickands constant. 
\par
The discrete analogs of Pickands's constant are connected to extremal 
indices of the Brown-Resnick processes.
Assume $d=1$ and consider a Brown-Resnick process $(\eta(t))_{t
  \in \bbr}$. Its restriction to the integers yields a strictly
stationary time series $(\eta(i))_{i\in\bbz}$.
For $x\in \R$ we have
\beao
\Pro \Big( \max_{i=1,\ldots,n} \eta(i) - \log n \le x\Big) = \exp \Big(-\ex^{-x} 
n^{-1}  \Expec\Big[ \max_{i=1,\ldots,n}\ex^{Z(i)}\Big]\Big)\,.
\eeao
This leads to the limit relation
\beao
\lim_{n\to\infty} 
\Pro \Big( \max_{i=1,\ldots,n} \eta(i) - \log n \le x\Big)=
\Lambda^\theta(x)\,,\qquad x\in\bbr\,,
\eeao
where the limit
\beao
\theta=\lim_{\nto}n^{-1}  \Expec \Big[\max_{i=1,\ldots,n} \ex^{Z(t)} \Big]
\eeao
exists by subadditivity and translation invariance as in the
continuous case. It is well known (see Leadbetter et
al.~\cite{leadbetter:lindgren:rootzen:1983}, cf.~Section 8.1 in
Embrechts et al.~\cite{embrechts:kluppelberg:mikosch:1997}) that
$\theta$ is a number in $[0,1]$.
The quantity $\theta$ is the {\em extremal index} of the
stationary \seq\ $(\eta(i))_{i\in\bbz}$. 
The reciprocal of this  quantity is often
interpreted as the expected value of the cluster size of
 high-level exceedances of the \seq\ $(X_i)$; see for example
\cite{leadbetter:lindgren:rootzen:1983}; cf.~ Section 8.1 in \cite{embrechts:kluppelberg:mikosch:1997}.
The constant $\theta$ appears in Dieker and Yakir~\cite{dieker:yakir:2013} 
as a special case of the constants $\eta {\mathcal H}_\alpha^\eta$; see Proposition 3 
in \cite{dieker:yakir:2013} 
for a characterization alternative {\blue ???} to the extremal index. Although we do not have a proof
that $\theta$ is smaller than Pickands's constant in the
continuous-time case, simulation evidence indicates that this fact is true.

\section{A simulation algorithm}\label{sec:alg}
This section presents a simulation algorithm for Brown-Resnick random fields on a discrete set of points $t_1,\ldots,t_n\in \R^d$.
We may assume that $\sigma^2/2=\gamma$ in this section.
Since Theorem~\ref{thm:main} gives a different representation for each choice of $\mu$,
it would be interesting to know which choice leads to the fastest algorithm.
Here we simply let $\mu$ be uniform on $\{t_1,\ldots,t_n\}$.

Remark~\ref{rem:simrep} shows that the vector $(N(t_1),\ldots,N(t_n))$ with, for  $j=1,\ldots,n$,
\[
N(t_j) = \sup_{i\ge 1} \left(V_i + W_i(t_j) -  \gamma(t_j-T_i)- \log\Big(n^{-1} \sum_{\ell=1}^n 
\exp(W_i(t_\ell) -  \gamma(t_\ell-T_i))\Big)\right)
\]
has the same distribution as $(\eta(t_1),\ldots,\eta(t_n))$,
where $\big((V_i,T_i)\big)_{i\ge 1}$  belong to a Poisson process on
$\bbr\times \{t_1,\ldots,t_n\}$ with intensity \ms\ $\ex^{-x}dx\times
\big(n^{-1}\sum_{i=1}^n \delta_{t_i}(dy)\big)$.
We slightly rewrite the above display as
\[
N(t_j) =\sup_{i\ge 1} \left(V_i+\log n + W_i(t_j) -  \gamma(t_j-T_i)- \log\Big(\sum_{\ell=1}^n 
\exp(W_i(t_\ell) -  \gamma(t_\ell-T_i))\Big)\right).
\]
This is the representation we use for our simulation algorithm.

A point $V_i$ on $\R$ gives rise to a `cluster' of points 
$\{C_i(t_j): j=1,\ldots,n\}$ with
\[
C_i(t_j) = 
(V_i+\log n) + W_i(t_j) -  \gamma(t_j-{T_i})-\log \Big(
\sum_{\ell=1}^n \exp(W_i(t_\ell) -  \gamma(t_\ell-{T_i}))\Big).
\] 
These cluster points can be visualized by interpreting them as belonging to different levels depending 
on the value of $j$; see Figure~\ref{fig:algorithm}. The variable $N(t_j)=\sup_{i\ge 1} C_i(t_j)$ 
is then the maximum of all cluster points on the $j$-th level.
The crucial insight is that {\em only a finite number of points/cluster pairs $(V_i,C_i)$ need to be generated, since
$C_i(t_j)\le V_i+\log n$ and we seek $\sup_{i\ge 1} C_i(t_j)$ for $j=1,\ldots,n$}.
The algorithm generates points/cluster pairs $(V+\log n,C)$ 
in decreasing order of $(V+\log n)$-value, until the next $(V+\log n)$-value 
is smaller than the current maximum over the cluster points on each level.
For instance, in Figure~\ref{fig:algorithm}, after $V_3+\log 4$ has been generated, none of the remaining
cluster points can change the values of $(N(t_1),\ldots,N(t_4))$, which have been given a different color.

%figure data
\iffalse
point =  2.7683
C =
   1.30840
   0.79795
   1.27598
   1.86075
point =  1.6363  % whole point/cluster pair shifted in picture to make it clearer
C =
   0.36469
   0.87234 % I moved it a bit on the picture to illustrate the key point better
  -0.34901
  -0.51385
point =  0.60112
\fi
\begin{figure}
\begin{center}
\includegraphics[width=\textwidth]{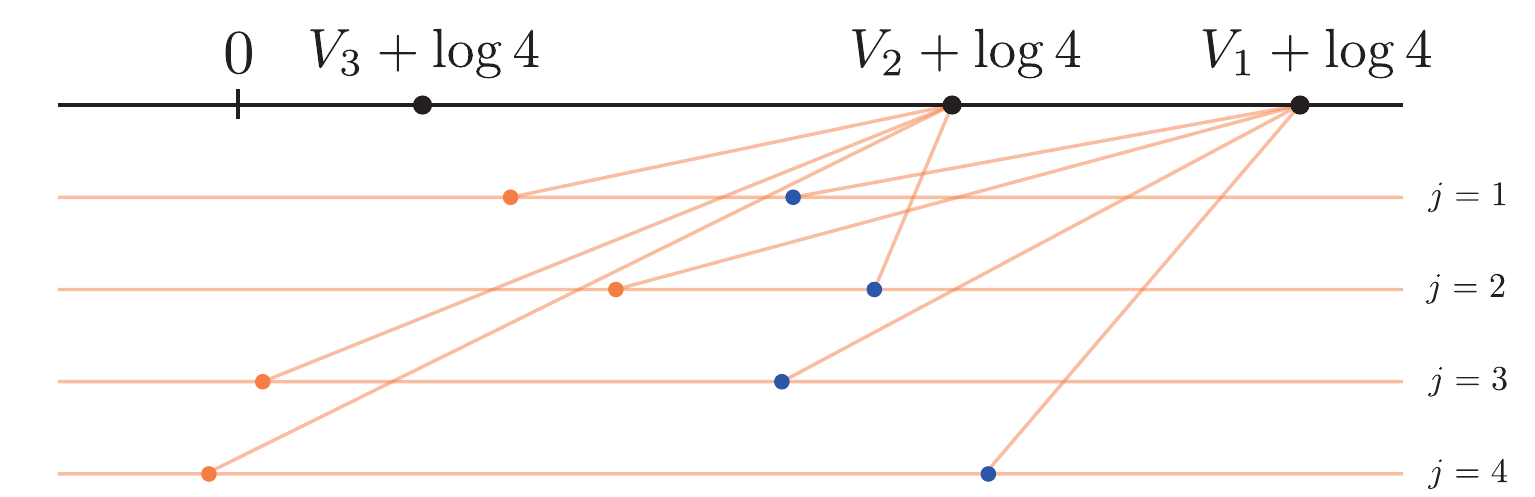}
\end{center}
\caption{
Illustration of our algorithm for $n=4$. 
The points $V_i+\log 4$ are generated in decreasing order.
Each `level' below the axis represents a value of $j$, and each $(V_i+\log 4)$-point
is connected to its cluster points $C_i(t_j)$.
The cluster points $C_i(t_j)$ always lie to the left of $V_i+\log 4$.}
\label{fig:algorithm}
\end{figure}

To get a sense of how many points of $V$ will be generated, let us consider the (degenerate) case 
where $t_1=\cdots=t_n=t$. We then have $C_i(t_j) = V_i$ for $j=1,\ldots,n$,
so the algorithm terminates after generating $\inf\{M: V_M+\log n< V_1\}$ points of $V$.
For large $n$, this implies that the number of points is of order $n$.

We remark that this algorithm is suitable for parallelization. 
Indeed, several points of the $V$-process can be generated simultaneously instead of one at the time,
with corresponding clusters being computed on different processors.
{\blue Specifically, with one master and $K$ workers, the algorithm would consist of a number of steps,
each of which computes the next $K$ clusters in parallel.
At each step of the algorithm, the master generates the next $K$
$V$-points in decreasing order. This is readily done since 
the points $(\ex^{-V_i})$ constitute a  standard Poisson process on $\R_+$. 
Each of the $K$ clusters would then be computed on a worker node,
after which the master checks whether the algorithm can be terminated or whether further steps are needed. }

\section{Numerical experiments}
This section reports on several simulation experiments we have carried out in order to validate our
algorithm and to test its performance in terms of speed.
Throughout, we work with Brown-Resnick random fields with variogram 
$\gamma(t)=|t|^\alpha/2$ for some $\alpha\in(0,2]$. 
Appendix~\ref{app:matlab} has some implementation details.

\subsection*{Representative samples}
We have implemented the algorithm in R (see \cite{R-project})
in order to leverage the existing toolkit to generate the Gaussian random fields that are needed in our algorithm.
We use the R package \verb|RandomFields| by Schlather et al., which is available
through R's package manager.
Three representative samples of Brown-Resnick random fields are given in Figure~\ref{fig:BRfield}, with
various levels of a {\blue smoothness} parameter $\alpha$. We see that the paths
become rougher as $\alpha$ decreases, as it should be.
The random field is the maximum of random `mountains' (given by quadratic forms) if $\alpha=2$,
and our replication for $\alpha=3/2$ exhibits similar behavior in the sense that two mountains can
be distinguished.
\begin{figure}
\begin{center}
\includegraphics[width=50mm]{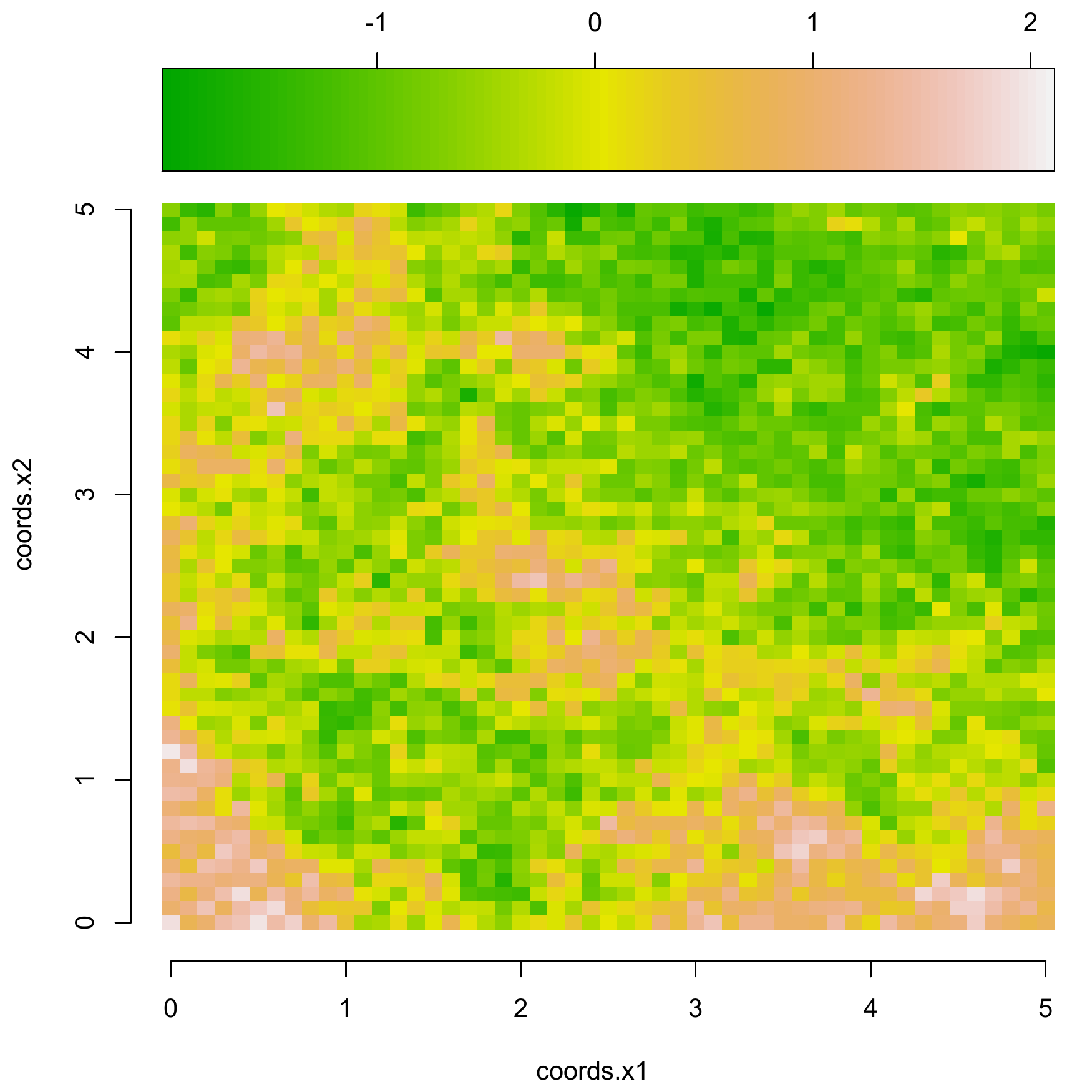}
\includegraphics[width=50mm]{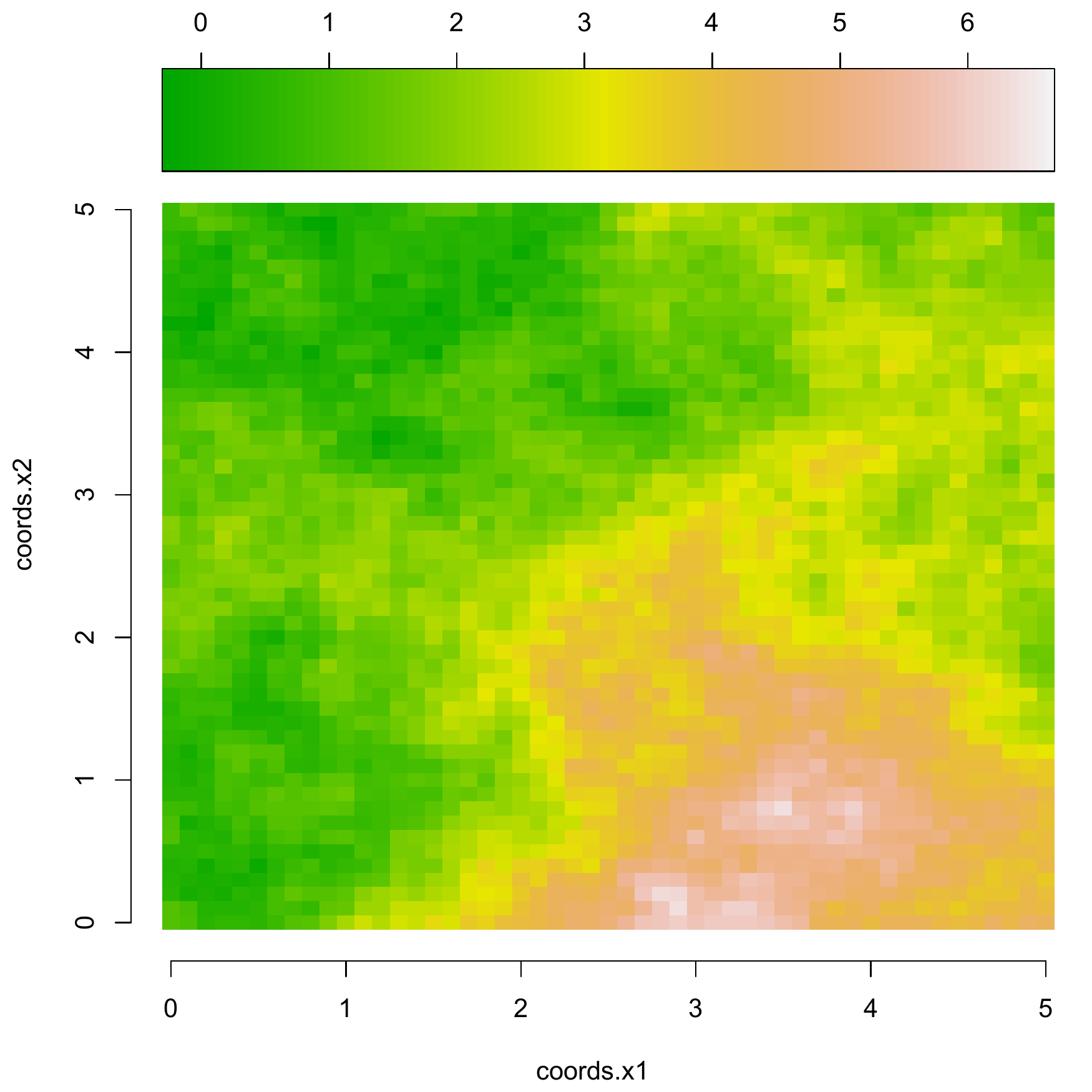}
\includegraphics[width=50mm]{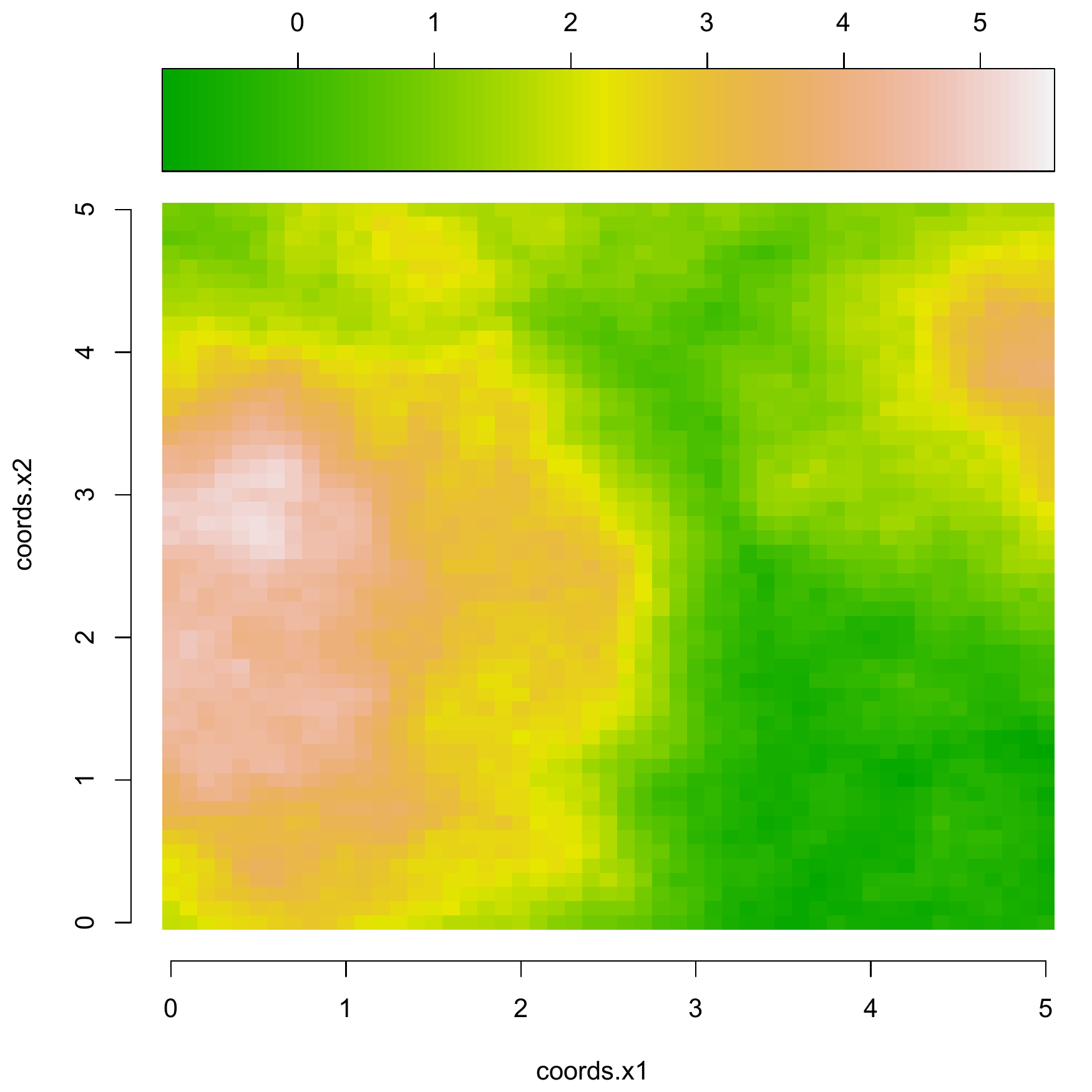}
\end{center}
\caption{Sample of a Brown-Resnick random field on $[0,5]^2$ with variogram $\gamma(t) = |t|^\alpha/2$
for $\alpha=1/2$, $\alpha=1$, $\alpha=3/2$ from left to right, respectively. The grid mesh is $0.1$.}
\label{fig:BRfield}
\end{figure}

\par
In the rest of this section, we carry out all experiments in
the one-dimensional case $d=1$ for computational ease.
Figure~\ref{fig:stationarity} depicts some representative one-dimensional samples for $\alpha=1$.
\begin{figure}
\begin{center}
\includegraphics[width=90mm]{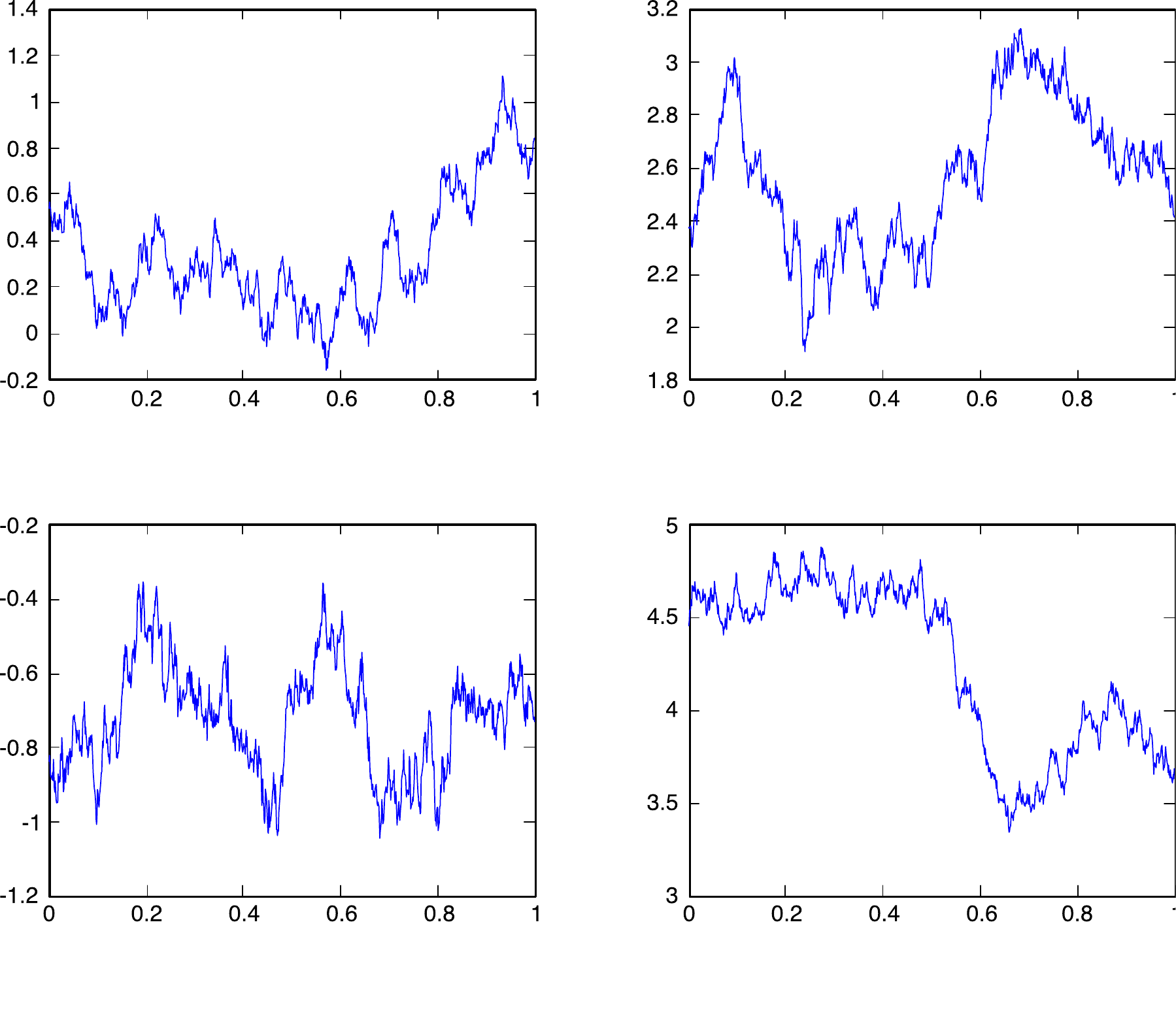}
\end{center}
\caption{Representative samples of a Brown-Resnick process on $[0,1]$ with variogram $\gamma(t)=|t|/2$. }
\label{fig:stationarity}
\end{figure}
Note that it indeed appears that these are realizations of a {\blue stationary} process even though 
our algorithm %avoids increasing 
does not require truncating the number of Gaussian random field samples
if one aims at a sample path of the process on a larger interval.

%We have generated 1000 iid replications of the vector $(\eta(0),\eta(1/1024),\ldots,\eta(1-1/1024))$ for $\alpha=1$.
%Figure~\ref{fig:qqplotmarginal} compares the empirical quantiles of $\eta(0)$ and $\eta(1-1/1024)$.
%Note that the distribution of both these variables is the standard Gumbel distribution, and the Q--Q plot supports that the simulated 
%distributions are indeed equal.
%\begin{figure}
%\begin{center}
%\includegraphics[width=60mm]{qqplot1dim.pdf}
%\end{center}
%\caption{A representative Q--Q plot illustrating the equality of the marginal distributions.}
%\label{fig:qqplotmarginal}
%\end{figure}

\subsection*{Dependence structure}
We next verify whether our simulation algorithm captures the dependence
within the process correctly. To this end, we generated 1000 samples of
$\eta(0)\vee \eta(s)$ for $\alpha=1$ in the one-dimensional case.
This random variable has a (nonstandard) Gumbel distribution.

We note that
\begin{eqnarray*}
\lefteqn{-\log \Pro(\eta(0)\vee \eta(s) \le x)}\\
 &=& \ex^{-x} \left[\Pro(W(s) \le s/2) + \Expec \big(\ex^{W(s)-s/2}; W(s)>s/2\big)\right]\\
&=& \ex^{-(x-\log(2\Phi(\sqrt{s}/2)))},
\end{eqnarray*}
where $\Phi$ is the distribution function of the standard normal distribution.
%Therefore, 
%\[
%\eta(0)\vee \eta(s) - \log(2\Phi(\sqrt{s}/2)) 
%\]
%has a standard Gumbel distribution. 
%
%The max-stable property of $\eta$ implies that $(\eta_1)_{t\in \R^d}$ and $(\max(\eta_2,\eta_3)-\log(2))_{t\in \R^d}$ 
%have the same distribution.
%Interpreting $\eta$ as a vector,
%we verify the equality of the distribution of $\|\eta_1\|$ and $\|\max(\eta_2,\eta_3)-\log(2)\|$, where
%$\eta_1,\eta_2,\eta_3$ are iid and the maximum is taken component-wise.
%Using our 1000 replications, we get a total of 333 samples from $\|\eta_1\|$ and 333 from $\|\max(\eta_2,\eta_3)-\log(2)\|$.
The resulting Q--Q plot for $s=1-1/1024$ is given in Figure~\ref{fig:qqplotdependence}.
\begin{figure}
\begin{center}
\includegraphics[width=80mm]{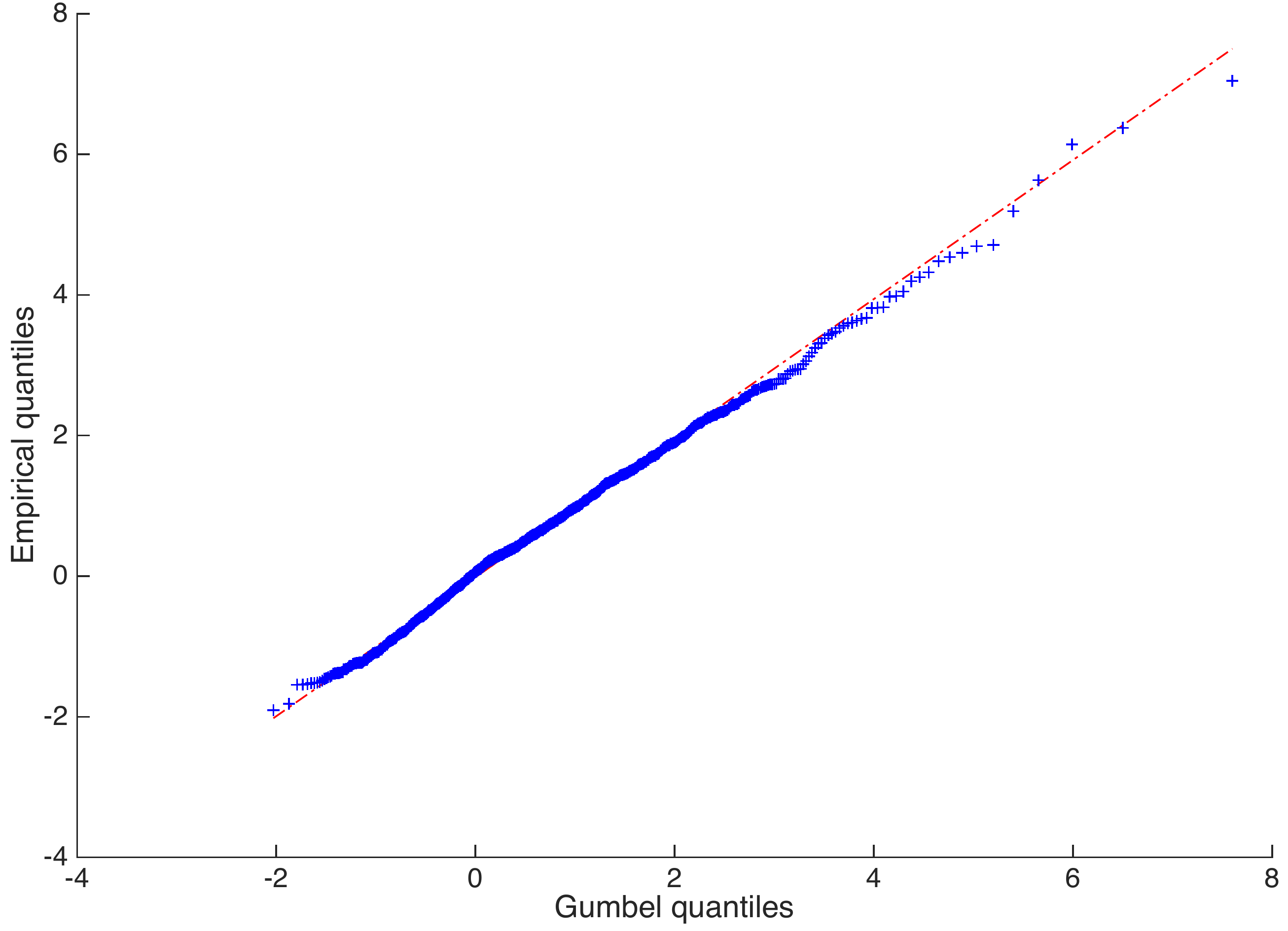}
\end{center}
\caption{Q--Q plot illustrating that our samples of $\eta(0)\vee \eta(s) -\log(2\Phi(\sqrt{s}/2))$ have a standard Gumbel distribution for $s=1-1/1024$.}
\label{fig:qqplotdependence}
\end{figure}

\subsection*{Number of clusters}
We next investigate numerically whether 
the dependency structure influences the number of points $V_i$ that are generated
by our algorithm for a single replication of the Brown-Resnick process. 
To do so, we generated 1000 replications of $(\eta(0),\eta(1/1024),\ldots,\eta(1-1/1024))$ 
for various values of $\alpha$.
Figure~\ref{fig:boxplotN} summarizes the results in a box plot.
\begin{figure}
\begin{center}
\includegraphics[width=96mm]{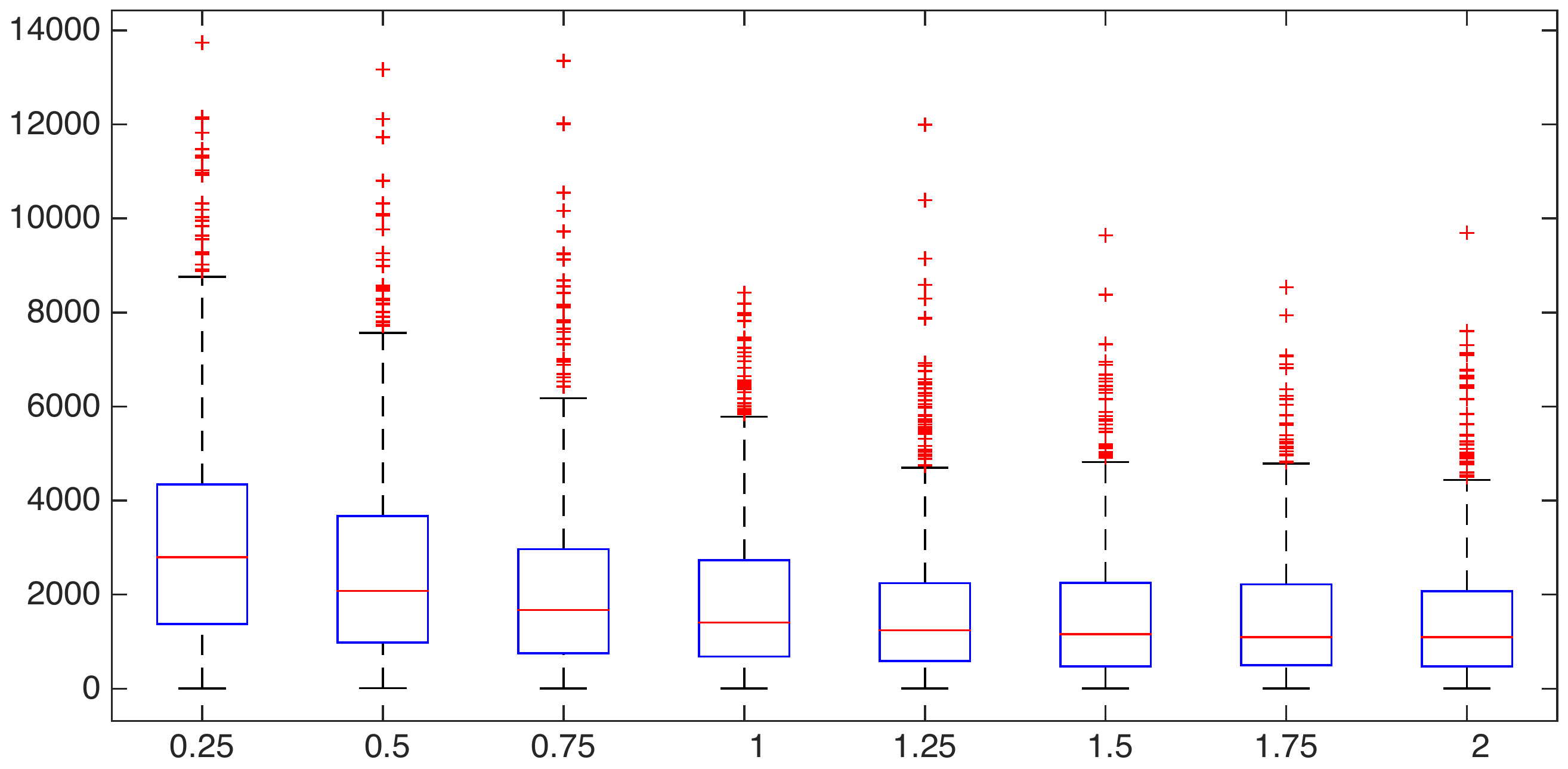}
\end{center}
\caption{Box plot showing the number of clusters generated as a function of $\alpha$, with $n=1024$.
The edges of the box are the 25th and 75th percentiles.}
\label{fig:boxplotN}
\end{figure}

The data provides evidence that rougher paths are harder to simulate,
which suggests that the order $n$ bound derived in 
Section~\ref{sec:alg} is in fact a lower
bound on the number of $V_i$ points that need to be generated.
In the code used for Figure~\ref{fig:boxplotN}, 
we preprocess some of the computations required for sampling the $W_i$.
This results in significant savings. 
We have not included this code in Appendix~\ref{app:matlab} for expository reasons.

We next compare the histogram of the number of $V_i$ points
for $\alpha=2$ with a fitted exponential density, 
see Figure~\ref{fig:histogram}.
\begin{figure}
\begin{center}
\includegraphics[width=76mm]{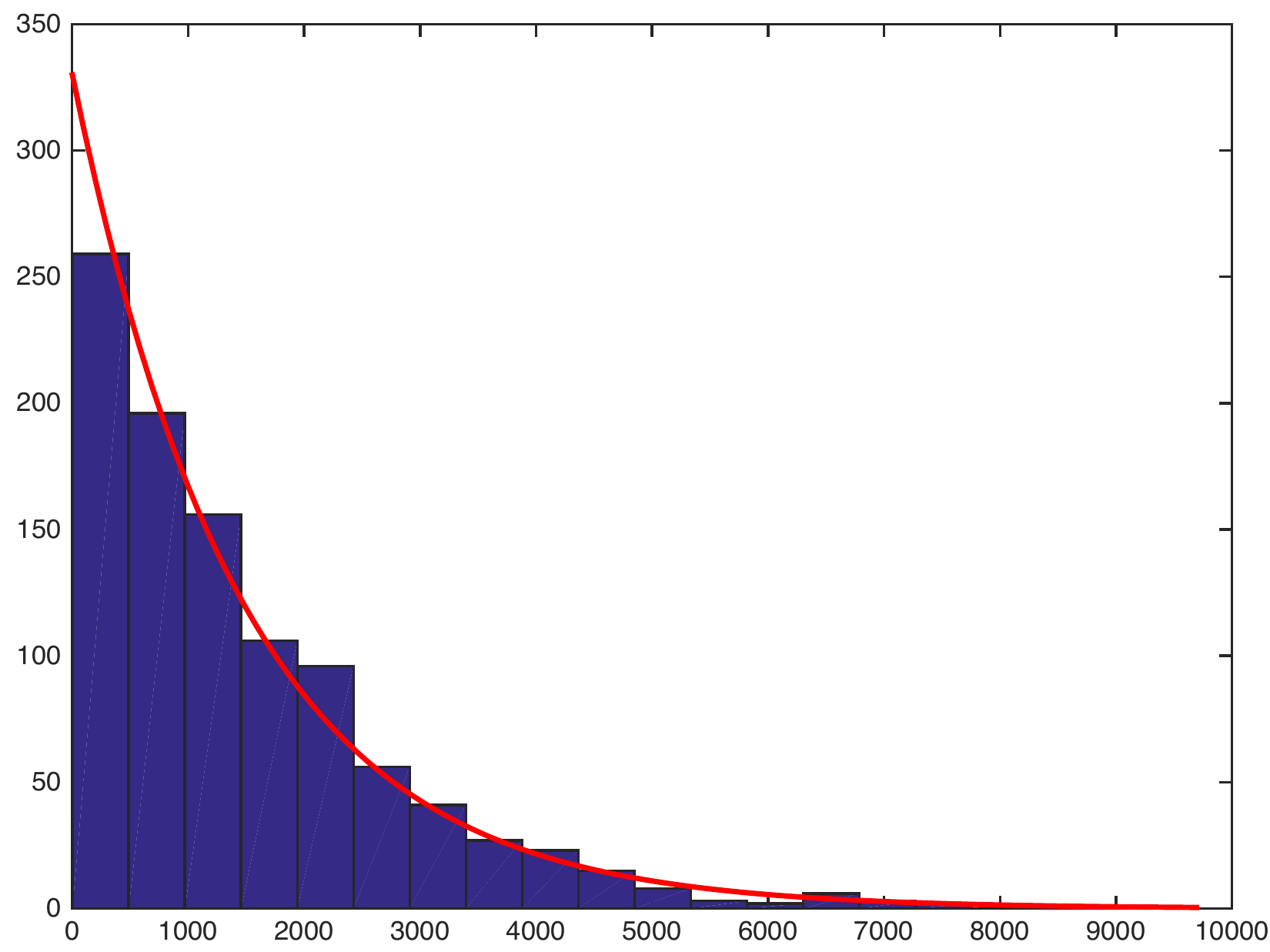}
\end{center}
\caption{Histogram of the number of $V_i$ points for $\alpha=2$, 
compared with an exponential density.}
\label{fig:histogram}
\end{figure}
This figure provides evidence that this distribution has light tails. 
For other values of $\alpha$, the corresponding histograms also 
indicate light tails, although the distribution looks more like a 
gamma distribution.

\subsection*{Speed}
The speed of our algorithm in practice heavily depends on how quickly the underlying
Gaussian random fields can be generated. In our one-dimensional case, 
we generate the Gaussian processes with the recent Matlab implementation 
by Kroese and Botev~\cite{kroese:botev:2013}. 
Theoretically, the computational effort needed to generate 
a sample from the underlying Gaussian process is independent of $\alpha$ in this implementation.
Thus, the running time depends linearly on the number 
of points $V_i$ that are generated by the algorithm, which is different for different samples.
 
In Matlab it is difficult to record CPU time (as opposed to elapsed time),
and we have observed wide variation (up to 50\%) in run time with exactly the same random input
on a dedicated CPU.
Thus, we keep the discussion at a high level. 
For the experiment reported in Figure~\ref{fig:boxplotN} with $n=1024$, 
each sample is generated in the order of seconds 
on a single core of a 2.7 GHz Intel Core i7 processor 
regardless the value of $\alpha$, with most runs less than a second and a few runs more than four seconds.

%We now investigate the relationship between the smoothness parameter $\alpha$ and
%the speed of the computer code. 
%Each sample is generated in the order of seconds regardless the value of $\alpha$.
%The results are summarized in the box plot in Figure~\ref{fig:boxplot}, which is based
%on the same experiment as Figure~\ref{fig:boxplotN}.
%\begin{figure}
%\begin{center}
%\includegraphics[width=90mm]{boxplot.pdf}
%\end{center}
%\caption{Box plot showing the dependence of $\alpha$ on the algorithm's running time (in seconds).
%The edges of the box are the 25th and 75th percentiles.}
%\label{fig:boxplot}
%\end{figure}

\section{Proofs}
\label{sec:proofs}
This section presents the proof of Theorem~\ref{thm:main}
We fix the functions $\sigma^2$ and $\gamma$ throughout this section. 
Contrary to the preceding two sections, we do not assume that $\gamma=\sigma^2/2$ but
we shall see that the function $\sigma^2$ vanishes from our calculations.

We start with an auxiliary lemma; 
%but we include it for completeness; 
{\blue see de Haan \cite{dehaan:1984} and Kabluchko et
al.~\cite{kabluchko:schlather:dehaan:2009} for proofs.} 
\ble
\label{lem:suplemma}
Let $(X_i)$ be iid copies of some random field $X$ on $\R^d$
and $(V_i)$ the points of a Poisson process on $\R$ with intensity measure $\ex^{-x}\,dx$.
If we write
\[
\xi(t) = \sup_{i\ge 1}\, (V_i + X_i(t))\,,\qquad t\in \R^d\,,
\]
then we have for $y_j\in\bbr\,,t_j\in \R^d\,,i=1,\ldots,n$,
\[
\Pro(\xi(t_1)\le y_1,\ldots, \xi(t_n)\le y_n) = \exp\Big(-\Expec \exp\Big(
  \max_{j=1,\ldots,n} (X(t_j) -y_j)\Big)\Big)\,.
\]
\ele

The following change of measure lemma plays a key role in our argument, and shows why
the variance function $\sigma^2$ vanishes from the calculations.
It is a field version of Lemma 1 in Dieker and Yakir~\cite{dieker:yakir:2013},
{\blue see also \cite[Prop.~2]{kabluchko:spectral:2009} for the underlying 
change of measure result}.
We only sketch the key idea of the proof insofar as it highlights the differences with
\cite{dieker:yakir:2013}, since the lemma follows from the same arguments as given there.

\begin{lemma}
\label{lem:changeofmeasure}
Fix $t\in \R^d$.
For a measurable functional $F$ on $(\R^d)^\R$ that is 
{\blue translation invariant}, 
%under addition of a constant function
we have
\[
\Expec \ex^{W(t)-\sigma^2(t)/2} F(W-\sigma^2/2) = \Expec F(\theta_t Z),
\]
where the shift $\theta_t$ is defined through $(\theta_t Z)(s) = Z(s-t)$.
\end{lemma}
\begin{proof}[Proof sketch]
Set $\Q (A) = \Expec[\ex^{W(t)-\sigma^2(t)/2} 1_A]$, and write $\Expec^\Q$ for the expectation operator with respect to $\Q$.
In this sketch, we first show that $W(s)-\sigma^2(s)/2$ under $\Q$ has the same distribution as $W(s)-\gamma(s-t)+\sigma^2(t)/2$ under $\Pro$.
The full proof requires doing this calculation for finite-dimensional distributions
to conclude that $(W(s)-\sigma^2(s)/2)_{s\in\R^d}$ under $\Q$ has the same distribution as $(W(s)-\gamma(s-t)+\sigma^2(t)/2)_{s\in \R^d}$ under $\Pro$,
but doesn't require additional insights.
%Since $2\Cov(B_s,B_t)=|s|^{2H}+|t|^{2H}-|s-t|^{2H}$, we deduce that
We compare generating functions: for any $\beta\in\R$, 
\begin{eqnarray*}
\lefteqn{
\log \Expec^\Q \exp\left(\beta (W(s)-\sigma^2(s)/2)\right)}\\
&=& -\frac 12\sigma^2(t) -\frac{\beta}{2} \sigma^2(s) +\frac 12 \Var \left[W(t) + \beta W(s) \right] \\
&=& -\frac{\beta}{2} \sigma^2(s)+\beta \Cov(W(t),W(s)) +\frac 12\Var \left[\beta W(s)\right] \\
&=& 
\beta \left[\frac12\sigma^2(t)-\gamma(s-t)\right]+\frac 12 \Var \left[\beta W(s)\right]\\
&=& 
\beta\Expec\left[W(s)-\gamma(s-t)+\frac 12\sigma^2(t) \right]+
\frac {\beta^2}2 \Var\left[W(s)-\gamma(s-t)+\frac 12\sigma^2(t)\right].
\end{eqnarray*}
Since $F$ is translation invariant, the $F$-value of $(W(s)-\gamma(s-t)+\sigma^2(t)/2)_{s\in \R^d}$ 
must be the same as the $F$-value of $(W(s)-W(t)-\gamma(s-t))_{s\in \R^d}$. 
The latter has the same distribution as $(Z(s-t))_{s\in\R^d}$, which yields the claim.
\end{proof}

\begin{proof}[Proof of Theorem~\ref{thm:main}]
Let $t_i\in\R^d\,, i=1,\ldots,n$ and $y_i\in\bbr\,,i=1,\ldots,n$ be arbitrary.
From Lemma~\ref{lem:suplemma} with $X_i=W_i-\sigma^2/2$ we deduce that
\[
\Pro (\eta(t_1)\le y_1,\ldots, \eta(t_n)\le y_n) = \exp\Big(-\Expec \exp\Big(
\max_{j=1,\ldots,n} (W(t_j) -\sigma^2(t_j)/2-y_j)\Big)\Big).
\]
Suppose that $\mu$ is an arbitrary probability measure on $\R^d$.
Applying Lemma~\ref{lem:changeofmeasure} with
\[
F(x) =  \frac{\max_{j=1,\ldots,n} \exp(x(t_j)-y_j) }{ \int_{\R^d} \exp(x(s)) \mu(ds) },
\]
we find that
\begin{eqnarray*}
\lefteqn{
\Expec \exp\Big( \max_{j=1,\ldots,n} (W(t_j) -\sigma^2(t_j)/2-y_j)\Big)}\\ &=& 
\int_{\R^d} \Expec\Big[ \exp(W(t)-\sigma^2(t)/2)
\frac{\exp\Big( \max_{j=1,\ldots,n} (W(t_j) -\sigma^2(t_j)/2-y_j)\Big)  }{\int_{\R^d} \exp(W(s) -\sigma^2(s)/2)\,\mu(ds)}\Big]\,\mu(dt)\\
&=&\int_{\R^d} \Expec \Big[\frac{\exp \big(\max_{j=1,\ldots,n} (Z(t_j-t) -y_j)\big) }{\int_{\R^d} \exp(Z(s-t))\,\mu(ds)}\Big]\,\mu(dt)\\
&=&\Expec \Big[\frac{\exp \big(\max_{j=1,\ldots,n} (Z(t_j-T) -y_j)\big) }{\int_{\R^d} \exp(Z(s-T))\,\mu(ds)}\Big],
\end{eqnarray*}
where {\blue the last expectation is taken \wrt\ the \ds\ of $(T,Z)$,
which is the product of the marginals.}
Applying Lemma~\ref{lem:suplemma} with 
\[
X_i(t) = Z_i(t-T_i) -\log\Big(\int_{\R^d} \exp(Z_i(s-T_i))\,\mu(ds)\Big)
\]
shows that 
\[
\Pro (\eta(t_1)\le y_1,\ldots, \eta(t_n)\le y_n)=
\Pro (\zeta(t_1)\le y_1,\ldots, \zeta(t_n)\le y_n).
\]
This yields the claim of Theorem~\ref{thm:main}.
\end{proof}

\section*{Acknowledgments}
The authors are grateful to the organizers of the workshop Stochastic Networks And Risk Analysis IV in Bedlewo, Poland, 
where much of this work was completed.
TM thanks Liang Peng for inviting him to Georgia Tech in April 2013,
when this work was initiated. 
ABD is supported in part by NSF CAREER grant CMMI-1252878,
and TM by DFF grant 4002-00435.
We thank the anonymous referees for their constructive comments and suggestions.

\appendix
\section{Computer code}
\label{app:matlab}
This Matlab code is for 1-dimensional parameter spaces, but it is almost immediately adaptable for use with random fields 
due to Matlab's capabilities to work with multidimensional arrays.
We present the Matlab code here since it can be read as pseudo-code, while
reading the R code requires some knowledge of R objects designed for spatial data.
%Our R implementation builds on a sophisticated toolbox for generating Gaussian random fields.

\begin{lstlisting}
function res = generate_cluster(n,V)
    T = floor(n*rand());
    W = generateWwithdriftandcenter(T); 
    res = V + W - log(sum(exp(W)));
end


function supremum = maxstable(n)
    supremum = -Inf(n,1);
    expminusV = -log(rand())/n;
    C = generate_cluster(n,-log(expminusV));

    while ( min(max(supremum, C)) < -log(expminusV) )
        supremum = max(supremum, C);    
        expminusV = expminusV - log(rand())/n;
        C = generate_cluster(n,-log(expminusV));
    end

    supremum = max(supremum, C);
end
\end{lstlisting}

\end{document}